\documentclass{preprint}

\usepackage{cancel}
\usepackage{geometry}
\usepackage{soul}
\usepackage[full]{textcomp}
\usepackage{newtxtext}
\usepackage[cal=boondoxo]{mathalfa}
\usepackage{colortbl} 
\usepackage{amsmath}
\usepackage{mhequ}
\usepackage{longtable}
\usepackage{booktabs}
\usepackage{centernot}
\usepackage{amsthm}
\usepackage{amssymb}
\usepackage{mathrsfs}
\usepackage{enumerate}
\usepackage{bbm}
\usepackage{marginnote}
\usepackage{scalerel}
\usepackage{tikz}
\usepackage{tikz-cd}
\usepackage{graphicx}
\usepackage{caption}
\usepackage{subcaption}
\usepackage{pdfpages}
\usepackage{geometry}

\numberwithin{equation}{section}

\newtheorem{thm}{Theorem}[section]

\newtheorem*{thm*}{Theorem}
\newtheorem*{prop*}{Proposition}
\newtheorem*{lem*}{Lemma}

\newtheorem{lem}[thm]{Lemma}
\newtheorem{prop}[thm]{Proposition}
\newtheorem{cor}[thm]{Corollary}
\newtheorem{defn}[thm]{Definition}
\newtheorem{rem}[thm]{Remark}

\newcommand\R{{\mathbb R}}

\newcommand{\EE}{\mathbb{E}}

\newcommand{\NN}{\mathbb{N}}
\newcommand{\PP}{\mathbb{P}}

\newcommand{\RR}{\mathbb{R}}
\renewcommand{\SS}{\mathbb{S}}

\newcommand{\ZZ}{\mathbb{Z}}

\newcommand{\cC}{\mathcal C}

\newcommand{\cF}{\mathcal F}

\newcommand{\cR}{\mathcal R}

\definecolor{darkred}{rgb}{0.9,0.1,0.1}

\renewcommand{\epsilon}{\varepsilon}

\begin{document}
\title{Emergence of fractional Gaussian free field correlations in subcritical long-range Ising models}
\author{Trishen Gunaratnam, Romain Panis}
\institute{\large \textbf{Universit\'e de Gen\`eve} \\ 
{trishen.gunaratnam@unige.ch} \\  {romain.panis@unige.ch}}

\maketitle

\begin{abstract}
    We study corrections to the scaling limit of subcritical long-range Ising models with (super)-summable interactions on $\mathbb Z^d$. For a wide class of models, the scaling limit is known to be white noise, as shown by Newman (1980). In the specific case of couplings $J_{x,y}=|x-y|^{-d-\boldsymbol{\alpha}}$, where $\boldsymbol{\alpha}>0$ and $|\cdot|$ is the Euclidean norm, we find an emergence of fractional Gaussian free field correlations in appropriately renormalised and rescaled observables. The proof exploits the exact asymptotics of the two-point function, first established by Newman and Spohn (1998), together with the rotational symmetry of the interaction.  
\end{abstract}

\section{Introduction}
In this article, we study the emergence of fractional Gaussian free field (FGFF) correlations in the scaling limit of long-range Ising models below the critical point. 

Let $\ZZ^d$ be the hypercubic lattice of dimension $d \geq 2$. For $\boldsymbol{\alpha}, \mathbf{C} > 0$, let $J=(J_{x,y})_{\lbrace x,y\rbrace\subset \mathbb Z^d}$ be a collection of ferromagnetic interactions defined by
\begin{equs}\label{eq: interaction}
J_{x,y}
:=
\frac{\bf{C}}{|x-y|^{d+\boldsymbol{\alpha}}},
\end{equs}
where $|\cdot|$ is the Euclidean norm\footnote{The choice of the Euclidean norm, which is spherically symmetric, is not just a technicality. See the discussion after Theorem \ref{thm: white noise} and also Remark \ref{rem: spherical}} on $\mathbb R^d$.

The Ising model with interaction $J$ is defined on finite subsets $\Lambda \subset \ZZ^d$ as follows. For $\sigma=(\sigma_x)_{x\in \Lambda}\in \lbrace \pm 1\rbrace^\Lambda$, define the Hamiltonian
\begin{equs}
    H_{\Lambda,h}(\sigma):=-\sum_{\lbrace x,y\rbrace\subset \Lambda}J_{x,y}\sigma_x\sigma_y-h\sum_{x\in \Lambda}\sigma_x.
\end{equs}
The finite volume Ising model $\langle \cdot \rangle_{\Lambda,\beta,h}$ at inverse temperature $\beta\geq 0$ and external field $h \in \R$ is the probability measure under which, for each $F:\lbrace \pm 1\rbrace^\Lambda\rightarrow \mathbb R$, 
\begin{equs}
    \langle F\rangle_{\Lambda,\beta,h}:=\dfrac{1}{Z(\Lambda,\beta,h)}\sum_{\sigma\in \lbrace \pm 1\rbrace^\Lambda}F(\sigma)\exp\left(-\beta H_{\Lambda,h}(\sigma)\right),
\end{equs}
where $Z(\Lambda,\beta,h):=\sum_{\sigma\in \lbrace \pm 1\rbrace^\Lambda}\exp\left(-\beta H_{\Lambda,h}(\sigma)\right)$ is the partition function of the model. 

Using Griffiths' inequalities \cite{griffiths1967correlations}, one can obtain the associated infinite volume Ising model with free boundary conditions by taking weak limits of $\langle \cdot \rangle_{\Lambda,\beta,h}$ as $\Lambda\nearrow \mathbb Z^d$. Denote the limiting probability measure by $\langle \cdot \rangle_{\beta,h}$.

It is well known that the Ising model exhibits a phase transition for the vanishing of the \textit{spontaneous magnetisation}. That is, if  
\begin{equs}
    m^*(\beta)
    :=
    \lim_{h\rightarrow 0^+}\langle\sigma_0\rangle_{\beta,h},
\end{equs}
then $\beta_c:=\inf\lbrace \beta>0,\text{ } m^*(\beta)>0\rbrace\in (0,\infty)$. It has been proved in \cite{fisher1967critical} that $\beta_c>0$, whilst the celebrated Peierls argument \cite{peierls1936ising} yields the bound $\beta_c<\infty$. Throughout the rest of the article, we assume that $h=0$.

We are interested in subcritical scaling limits, i.e. when $\beta < \beta_c$. As recalled in the proposition below, in this regime these models still exhibit polynomial decay of spin-spin correlations and thus may have a non-degenerate scaling limit.

\begin{prop}[\cite{newman1998shiba,Ao}]\label{prop: sharp asymptotics} Let $d\geq 2$, $\boldsymbol{\alpha},\mathbf{C}> 0$ in the interaction \eqref{eq: interaction}, and $\beta<\beta_c$. Then, there exists $\delta > 0$ such that for every $x\in \mathbb Z^d\setminus \lbrace 0\rbrace$,
\begin{equs} \label{eq: exact asymptotic}
    \langle \sigma_0\sigma_x\rangle_\beta
    =
    \dfrac{\beta\chi(\beta)^2}{2}\dfrac{\mathbf{C}}{|x|^{d+\boldsymbol{\alpha}}} \Big( 1 + O\big(|x|^{-\delta}\big) \Big),
\end{equs}
where $\chi(\beta):=\sum_{x\in \mathbb Z^d}\langle \sigma_0\sigma_x\rangle_\beta$. As an immediate consequence, we obtain polynomial decay of correlations: there exists $C=C(\beta,J)>0$ such that for all $x\in \mathbb Z^d\setminus \lbrace 0\rbrace$,
\begin{equs} \label{eq: spin decay}
    \langle \sigma_0\sigma_x\rangle_\beta
    \leq
    \frac{C}{|x|^{d+\boldsymbol{\alpha}}}.
\end{equs}
\end{prop}

The relevant object to study the scaling limit is the set of smeared observables.  Let $\mathcal{C}^\infty_c(\mathbb R^d)$ be the set of infinitely differentiable and compactly supported functions from $\mathbb R^d$ to $\mathbb R$.
\begin{defn}[Smeared observable] Let $\beta>0$ and $L\in \mathbb{N}^*$. For every $f \in \mathcal{C}^\infty_c(\mathbb R^d)$, we define the smeared observable $T_{f,L,\beta}: \{\pm 1\}^{\ZZ^d} \rightarrow \mathbb{R}$ by,
\begin{equs}
    T_{f,L,\beta}(\sigma)
    :=
    \dfrac{2^{d/2}}{\sqrt{\Sigma_L(\beta)}}\sum_{x\in \mathbb Z^d}f\left(\dfrac{x}{L}\right)\sigma_x, \qquad \sigma \in \lbrace \pm 1\rbrace^{\mathbb Z^d}
\end{equs}
where 
\begin{equs}
    \Sigma_L(\beta)
    =
    \Big\langle \Big(\sum_{x\in \Lambda_L}\sigma_{x}\Big)^2\Big\rangle_\beta=\sum_{x,y\in \Lambda_L}\langle\sigma_x\sigma_y\rangle_\beta, \qquad \textup{ with }\Lambda_L:=[-L,L]^d\cap \mathbb Z^d.
\end{equs}
\end{defn}
\begin{rem} 
The natural scaling is guessed from considering the test function  $f=\mathbf{1}_{[-1,1]^d}$. In this case, we find that
\begin{equs}
    \langle T_{f,L,\beta}(\sigma)^2\rangle_\beta =1.
\end{equs}
More generally, for positive functions $f$ that are not identically zero, we have $0<c_f\leq \langle T_{f,L,\beta}(\sigma)^2\rangle_\beta \leq C_f<\infty$.
\end{rem}

It is classical \cite{A,frohlich1982triviality} that when $\beta<\beta_c$, for each $f \in \cC_c^\infty(\RR^d)$, any limit point in distribution of the random variables $T_{f,L,\beta}$ is necessarily a centred Gaussian random variable (see also Proposition \ref{prop: triviality}). Our first result identifies the variance as $\int_{\RR^d} f(\tilde x)^2 \mathrm{d}\tilde x$, hence proving full convergence. This implies that the set of averages, viewed as a stochastic process indexed by $f \in \cC_c^\infty(\RR^d)$, converges in the sense of finite dimensional distributions to the standard white noise process on $\RR^d$. This is somewhat counter-intuitive to the polynomial decay of spin-spin correlations since the white noise is a degenerate limit. 

The same result was already obtained by Newman \cite{newman1980normal,newman1983general} in a more general setup. Our approach is more specialised to the Ising model.

\begin{defn}
The white noise is a stochastic process $W_0=\{ W_0(f): f \in \mathcal{C}^\infty_c(\RR^d) \}$ defined on some probability space  $(\Omega, \cF, \PP)$ such that: i) each $W_0(f)$ is a centred, Gaussian random variable; ii) the mapping $f \mapsto W_0(f)$ is linear; and, iii) the covariance is given by,
\begin{equs}
\EE[ W_0(f) W_0(g)]
=
\int_{\mathbb R^d} f(\tilde x) g(\tilde x) \mathrm{d}\tilde x, 
\qquad \forall f,g \in \cC_c^\infty(\RR^d).
\end{equs}
\end{defn}

\begin{rem}
The existence of white noise follows by the Kolmogorov extension theorem. Note that $W_0$ may be interpreted almost surely as a random element of $S'(\RR^d)$, the space of Schwartz distributions. Alternatively, we may canonically extend $W_0$ to a stochastic process indexed by $L^2(\RR^d)$--- this is the Gaussian Hilbert space point of view. In doing so, we lose the interpretation as an almost sure element of a space of distributions, i.e. we lose the almost sure linearity property. 
\end{rem}

\begin{thm}[Convergence to the white noise]\label{thm: white noise}
Let $d \geq 2$ and $\boldsymbol{\alpha}, \mathbf{C} > 0$ in the interaction \eqref{eq: interaction}, and $\beta < \beta_c$. Then, the stochastic process $\{ T_{f,L,\beta}(\sigma) : f \in \mathcal{C}^\infty_c(\RR^d) \}$ converges in the sense of finite dimensional distributions to the standard white noise. Moreover, we have the following quantitative rate of convergence: for every $f\in \mathcal{C}_c^\infty(\mathbb R^d)$ and $z\in \mathbb R$, there exists $C=C(\beta,d,J,z,f)>0$, such that for $L\geq 1$,
\begin{equs}\label{eq: quantitative wn cv}
\left|\langle \exp\left(z T_{f,L,\beta}(\sigma)\rangle_\beta\right) - \exp\left(\frac{z^2}2 \int_{\RR^d} f(\tilde x)^2 \mathrm d\tilde x \right)\right|
\leq
CL^{-\frac{\boldsymbol{\alpha}}{\boldsymbol{\alpha}+1}}.
\end{equs}

\end{thm}

\begin{rem} Theorem \textup{\ref{thm: white noise}} (and in particular the polynomial rate of convergence) holds for translation-invariant Ising models on $\mathbb Z^d$ satisfying the following (super)-summable condition: for all $\beta<\beta_c$, there exists $\varepsilon,C>0$ such that for all $x\in \mathbb Z^d\setminus \lbrace 0\rbrace$,
\begin{equs}
\langle \sigma_0\sigma_x\rangle_\beta\leq \frac{C}{|x|^{d+\varepsilon}}.
\end{equs}
This class of models includes finite-range Ising models (see \textup{\cite{aizenman1987phase,DCT}}) as well as a wide class of ferromagnetic long-range interactions $(J_{x,y})_{x,y \in \ZZ^d}$. The former follows from the fact that the polynomial bound on spin-spin correlations in Proposition \textup{\ref{prop: sharp asymptotics}} is satisfied in greater generality (see \textup{\cite{newman1998shiba,Ao,aoun2023two}}).
\end{rem}
\begin{rem} The proof below can be adapted to show convergence to white noise of any subcritical translation-invariant Ising models on $\mathbb Z^d$. Indeed, the key point in the proof is the finiteness of the susceptibility.
\end{rem}

In the proof of Theorem \ref{thm: white noise}, we explicitly compute the limiting covariance structure of smeared observables. Due to the sharp asymptotics of the spin-spin correlation functions, lower order contributions to this covariance structure can also be computed. The spherical symmetry of $J$ allows us to find emergent fractional GFF correlations appearing for certain rescaled and renormalised product observables. Since Theorem \ref{thm: white noise} is a central limit theorem for the set of smeared observables, these higher order contributions can be viewed as corrections to the CLT.

\begin{defn}\label{def: fgff}
Let $\boldsymbol{\alpha}>0$. The fractional GFF of Hurst parameter $H=H(\boldsymbol{\alpha})=-d/2-\boldsymbol{\alpha}/2$ is a stochastic process $W_{\boldsymbol{\alpha}} = \{ W_{\boldsymbol{\alpha}}(f) : f \in \mathcal{C}^\infty_c (\RR^d) \}$ defined on some probability space $(\Omega, \cF, \PP)$ such that: i) each $W_{\boldsymbol{\alpha}}(f)$ is a centred, Gaussian random variable; ii) the mapping $f \mapsto W_{\boldsymbol{\alpha}}(f)$ is linear; and, iii) the covariance is given by, $
\forall f,g \in \mathcal{C}^\infty_c(\RR^d)$,
\begin{equs}
\EE[W_{\boldsymbol{\alpha}}(f) W_{\boldsymbol{\alpha}}(g)]
=
K_{\boldsymbol{\alpha},d}(f,g)
\end{equs}
where, for $\boldsymbol{\alpha} \in \RR_{> 0} \setminus 2\NN^*$,
\begin{equs}
K_{\boldsymbol{\alpha},d}(f,g)
&:=
\int_{\RR^d \times \RR^d} C(\boldsymbol{\alpha},d) |\tilde x-\tilde y|^{-(d+\boldsymbol{\alpha})} \Big( f(\tilde x)g(\tilde y) - \sum_{j=0}^{\lfloor\boldsymbol{\alpha}\rfloor_2/2} f(\tilde x) H_j \Delta^j g(\tilde x) |\tilde x -\tilde y|^{2j} \Big)\mathrm{d} \tilde x\mathrm{d}\tilde y, 
\end{equs}
and for $\boldsymbol{\alpha}\in 2\NN^*$,
\begin{equs}
 K_{\boldsymbol{\alpha},d}(f,g)
:=
(4\pi^2)^{-\boldsymbol{\alpha}/2}\int_{\RR^d \times \RR^d} f(\tilde x) (-\Delta)^{\boldsymbol{\alpha}/2} g(\tilde x) \mathrm{d}x.
\end{equs}
Above,  for $t>0$, $\lfloor t\rfloor_2:=\max\lbrace k\in 2\mathbb N, \: k<t\rbrace$, $\Delta^j$ is the $j$-th power of the Laplacian, and
\begin{align}
\begin{split}\label{def: H_j} 
C(\boldsymbol{\alpha},d)
&=
\frac{2^{\boldsymbol{\alpha}} \pi^{-d/2}\Gamma(d/2+\boldsymbol{\alpha}/2)}{\Gamma(-\boldsymbol{\alpha}/2)}, \qquad \boldsymbol{\alpha} \in \RR_{> 0} \setminus 2\NN^*, 
\\
H_j
&=
\frac{2\pi^{d/2}/\Gamma(d/2)}{2^j j! d (d+2) \dots (d+2j-2)}.
\end{split}
\end{align}
\end{defn}

\begin{rem}
\begin{enumerate}
\item[(a)] Fractional GFFs are surveyed in great generality and detail in the review paper \textup{\cite{lodhia2016fractional}}. When $\boldsymbol{\alpha}>0$, these random fields have singular short-range correlation structure. The case $\boldsymbol{\alpha} = 1$ in $d=3$ arises, for example, in the cosmology of the early universe \textup{\cite{dodelson2020modern}}. Note that this range of parameters does not include ``usual'' GFF, which corresponds to $\boldsymbol{\alpha}=-2$, or the fractional GFFs that usually arise from critical scaling limits of some long-range Ising models (see \textup{\cite{AF,chen2015critical}}).
\item[(b)] There is a subtle but important distinction between the cases $\boldsymbol{\alpha} \in \RR_{>0} \setminus 2\NN^*$ and $\boldsymbol{\alpha} \in 2\NN^*$. In the case $\boldsymbol{\alpha} \in 2\mathbb N^*$, the short-distance correlations are singular, but retain some locality because powers of the Laplacian are local. On the other hand, in the case $\boldsymbol{\alpha} \in \RR_{> 0} \setminus 2\NN^*$, the short-distance correlations are both singular and non-local. 
\item[(c)] The fractional GFF can be viewed as a fractional Laplacian of the white noise $W_0$. More precisely, the fractional GFF of Hurst parameter $H(\boldsymbol{\alpha})$ can be seen as $(-\Delta)^{\boldsymbol{\alpha}/2}W_{0}$. Due to the non-integrable singularity at the origin, the kernel associated with the fractional Laplacian must in general be interpreted as a renormalised distribution. The details can be found in \textup{\cite[Chapter 1]{landkof1972foundations}}.
\end{enumerate}
\end{rem}

We now introduce our renormalised smeared observables. 
\begin{defn}
Let $d \geq 2$, $\boldsymbol{\alpha}, \mathbf{C} > 0$ in the interaction \eqref{eq: interaction}, and $\beta < \beta_c$. Let $L\geq 1$. For $f,g \in \cC^\infty_c(\RR^d)$, and $\sigma \in \{ \pm 1 \}^{\ZZ^d}$, define
\begin{equs}
\cR_{\boldsymbol{\alpha},\beta,L}&\left[T_{f,L,\beta}(\sigma) T_{g,L,\beta}(\sigma)\right]
:=
\frac{2^d}{\Sigma_L(\beta)}\sum_{x,y\in \mathbb Z^d}f(x/L)\left(g(y/L)- {\rm Tay}_{\lfloor\boldsymbol{\alpha}\rfloor_2} g(y/L;x/L) \right) \sigma_x\sigma_y,
\end{equs}
where, for any $\tilde x,\tilde y \in \RR^d$ and any $h \in \cC^\infty_c(\RR^d)$,
\begin{equs}
{\rm Tay}_{\lfloor\boldsymbol{\alpha}\rfloor_2} h (\tilde y; \tilde x)
:=
\sum_{j=0}^{\lfloor \boldsymbol{\alpha} \rfloor_2} \frac{1}{j!}\nabla^j h(\tilde x) \underbrace{(\tilde y-\tilde x, \dots, \tilde y - \tilde x)}_{j \textup{-tuple}},
\end{equs}
and $\nabla^j h(\tilde x)$ denotes the $j$-th Fr\'echet derivative of $h$ at $\tilde x$.
\end{defn}

The following is our main result. 
\begin{thm}[Emergence of fractional GFF correlations] \label{thm: fgff}
Let $d \geq 2$, $\boldsymbol{\alpha}, \mathbf{C} > 0$ in the interaction \eqref{eq: interaction}, and $\beta < \beta_c$. Then, for all $f,g \in \cC^\infty_c(\RR^d)$,
\begin{equs}
\lim_{L \rightarrow \infty} L^{\boldsymbol{\alpha}} \Big\langle \cR_{\boldsymbol{\alpha},\beta,L}\left[T_{f,L,\beta}(\sigma) T_{g,L,\beta}(\sigma)\right]\Big\rangle_\beta
&=
\frac{\beta\chi(\beta)\mathbf{C}}{2C(\boldsymbol{\alpha},d)}\EE[W_{\boldsymbol{\alpha}}(f) W_{\boldsymbol{\alpha}}(g)], \qquad \boldsymbol{\alpha} \in \RR_{\geq 0}\setminus 2\NN
\\
\lim_{L \rightarrow \infty} \frac{L^{\boldsymbol{\alpha}}}{\log L} \Big\langle \cR_{\boldsymbol{\alpha},\beta,L}\left[T_{f,L,\beta}(\sigma) T_{g,L,\beta}(\sigma)\right]\Big\rangle_\beta
&=
\frac{\beta \chi(\beta)\mathbf{C}H_{\boldsymbol{\alpha}/2}}{2(4\pi^2)^{-\boldsymbol{\alpha}/2}}(-1)^{\boldsymbol{\alpha}/2}\EE[W_{\boldsymbol{\alpha}}(f) W_{\boldsymbol{\alpha}}(g)], \qquad \boldsymbol{\alpha} \in 2\NN^*.
\end{equs}
\end{thm}

\begin{rem}
Quantitative error estimates are given in Propositions \textup{\ref{prop: case 1}} and \textup{\ref{prop: case 2}}.
\end{rem}

Theorem \textup{\ref{thm: fgff}} appears to suggest that, for $\boldsymbol{\alpha} \in (0,2)$, there exists a copy of white noise $W_0$ and an independent FGFF $W_{\boldsymbol{\alpha}}$ such that, for some explicit constant $C>0$, for every $f \in \cC^\infty_c(\RR^d)$, and for every $L \in \NN$, there exists a random variable $F_{f,L}$ such that
\begin{equs}\label{eq: expression tf}
T_{f,L,\beta} 
\overset{\textup{law}}{=}
W_0(f)  + C^{-1/2}L^{-\boldsymbol{\alpha}/2} W_{\boldsymbol{\alpha}}(f)  + L^{-\boldsymbol{\alpha}/2}F_{f,L},
\end{equs}
and $F_{f,L}$ converges to $0$ as $L \rightarrow 0$ in distribution. Similarly, for general $\boldsymbol{\alpha}>0$, Theorem \ref{thm: fgff} suggests that a similar expansion holds, albeit with additional terms corresponding to derivatives of white noise. 

The following corollary establishes this rigorously for $\boldsymbol{\alpha} \in (0,2)$. It is a straightforward consequence of the quantitative triviality estimate of Proposition \ref{prop: triviality}, the quantitative error bounds in Propositions \ref{prop: case 1} and \ref{prop: case 2}, and sufficiently good rate\footnote{For $f \in \cC^\infty_c(\RR^d)$, the Riemann sum can be approximated up to an error that is $O(L^{-p})$ for any $p > 0$. This is a consequence of Poisson summation and the decay of the the Fourier transform of $f$ being faster than any polynomial.} of convergence of Riemann sums for smeared observables. The analogous result for general $\boldsymbol{\alpha}$ can be established similarly, with the aforementioned additional terms/scalings incorporated. We omit this for clarity of exposition. 

\begin{cor}\label{cor: fgff conv} Let $d\geq2$, $\boldsymbol{\alpha}\in (0,2), \mathbf{C}>0$, and $\beta < \beta_c$. For every $f\in \cC_c^\infty(\mathbb R^d)$, and every $z\in \mathbb R$, as $L\rightarrow \infty$,
\begin{equs}
\langle \exp\left(zT_{f,L,\beta}(\sigma)\right)\rangle_\beta
=
\exp\left(\frac{z^2}{2}\mathbb E[W_0(f)^2]\right)\exp\left(\frac{z^2}{2}\frac{1}{L^{\boldsymbol{\alpha}}}\frac{\beta\chi(\beta)\mathbf{C}}{2C(\boldsymbol{\alpha},d)}\mathbb E[W_{\boldsymbol{\alpha}}(f)^2]\right)(1+o(1)).
\end{equs}
As a result, $T_{f,L,\beta}$ can be written as in \textup{\eqref{eq: expression tf}}, with $F_{f,L}$ converging to $0$ in law.
\end{cor}

\subsection*{Notations}

We write $|\cdot|_1, |\cdot|,$ and $|\cdot|_\infty$ to denote the $1$-norm, Euclidean norm, and $\infty$-norm on $\RR^d$, respectively. We also write these symbols for the corresponding norms on $\ZZ^d$.

In continuum, we use the following conventions. We write $B(\tilde y,r) = \{ \tilde x \in \RR^d : |\tilde x-\tilde y| \leq r \}$ to denote the ball of radius $r>0$ and centred at $\tilde y\in\RR^d$. We write $S(\tilde y,r)=\partial B(\tilde y,r)$ to denote the $(d-1)$-sphere of radius $r$ centred at $\tilde y$. On the lattice, for $x \in \ZZ^d$ and $L \in \NN$, we write $\Lambda_L(x):=\{ y \in \ZZ^d : |x-y|_\infty \leq L \}$ to denote the box of side length $L$ centred at $x$.

If $(a_n)_{n\geq 0},(b_n)_{n\geq 0}\in (\mathbb R_+^*)^{\mathbb N}$, we write $a_n=O(b_n)$ (resp. $a_n=o(b_n)$) if there exists $C=C(d,\beta,J)>0$ such that for all $n\geq 1$, $a_n\leq C b_n$ (resp. $\lim_{n\rightarrow \infty} a_n/b_n=0$). Throughout the paper, the constants $C,C_i>0$ only depend on the parameters $d,\beta,J$, unless stated otherwise.

\subsection*{Acknowledgements}
We thank Ajay Chandra, Hugo Duminil-Copin, Antti Knowles, and Sylvia Serfaty for stimulating discussions. We also thank Yacine Aoun and Yvan Velenik for references. TSG acknowledges the support of the Simons Foundation, Grant 898948, HDC. RP was supported by the Swiss National Science Foundation and the NCCR SwissMAP.

\section{Convergence to white noise} \label{sec: white noise}

In this section, we prove Theorem \ref{thm: white noise}. We begin by showing the triviality of the smeared observables in the scaling limit. Then, we explicitly compute the covariance structure to fully characterise the limit as white noise.

\subsection{Triviality of any limit point}

The following proposition establishes that any limit point of a smeared observable is Gaussian in the subcritical phase, with quantitative bounds on the error in the Laplace transform.

\begin{prop}\label{prop: triviality} Let $d \geq 2$,  $\boldsymbol{\alpha}, \mathbf{C} > 0$ in the interaction \eqref{eq: interaction}, and $\beta < \beta_c$. For all $f\in \mathcal{C}^\infty_c(\mathbb R^d)$ and $z\in \mathbb R$, there exists $C=C(d,\beta,J,f)>0$ such that, for all $L\geq 1$,
\begin{equs}
    \left|\left\langle \exp\left(z T_{f,L,\beta}(\sigma)\right)\right\rangle_\beta-\exp\left(\frac{z^2}{2}\langle T_{f,L,\beta}(\sigma)^2\rangle_\beta\right)\right|
    \leq
    \frac{C}{L^d}.
\end{equs}
\end{prop}
\begin{proof} Recall that the 4-point Ursell function $U_4^\beta$ is defined for $x,y,z,t\in \mathbb Z^d$ by
\begin{equs}
U^\beta_4(x,y,z,t):=\langle \sigma_x\sigma_y\sigma_z\sigma_t\rangle_{\beta}-\langle \sigma_x\sigma_y\rangle_\beta\langle \sigma_z\sigma_t\rangle_\beta-\langle \sigma_x\sigma_z\rangle_\beta\langle \sigma_y\sigma_t\rangle_\beta-\langle \sigma_x\sigma_t\rangle_\beta\langle \sigma_y\sigma_z\rangle_\beta.
\end{equs}
Using \cite[Proposition~12.1]{A} as in \cite[Section~6.3]{ADC}, we get that for every $L\geq 1$,
\begin{equs}
    \left|\left\langle \exp\left(z T_{f,L,\beta}(\sigma)\right)\right\rangle_\beta-\exp\left(\frac{z^2}{2}\langle T_{f,L,\beta}(\sigma)^2\rangle_\beta\right)\right|
    %\\
    \leq C_1 z^4\exp\left(\frac{z^2}{2}\langle T_{|f|,L,\beta}(\sigma)^2\rangle_\beta\right)\Vert f\Vert_\infty^4S(\beta,L,f),
\end{equs}
where
\begin{equs}
    S(\beta,L,f)=\sum_{x_1,x_2,x_3,x_4\in \Lambda_{r_f L}}\dfrac{\left|U_4^\beta(x_1,x_2,x_3,x_4)\right|}{\Sigma_L(\beta)^2},
\end{equs}
$r_f=\max( r\geq 0, \: \exists x\in \mathbb R^d, \: |x|=r, \: f(x)\neq 0) \vee 1$, and $C_1>0$. The tree diagram bound \cite[Proposition 5.3]{A} asserts
\begin{equs}\label{tree diagram bound trivia subcritic}
    |U_4^\beta(x_i,x_j,x_k,x_\ell)|\leq 2\sum_{x\in \mathbb Z^d}\langle \sigma_x\sigma_{x_i}\rangle_\beta\langle \sigma_x\sigma_{x_j}\rangle_\beta\langle \sigma_x\sigma_{x_k}\rangle_\beta\langle \sigma_x\sigma_{x_\ell}\rangle_\beta.
\end{equs}
Splitting the resulting sum,
\begin{equs}
    S(\beta,L,f)/2
    \leq
    \underbrace{\sum_{\substack{x \in \Lambda_{2r_{f}L} \\ x_1,x_2,x_3,x_4\in \Lambda_{r_{f}L}}}(\ldots)}_{(\mathrm{I})}+\underbrace{\sum_{\substack{x \notin \Lambda_{2r_{f}L} \\ x_1,x_2,x_3,x_4\in \Lambda_{r_{f}L}}}(\ldots)}_{(\mathrm{II})}.
\end{equs}
\textbf{Bound on (I)}. The first term can be written
\begin{equs}
    \sum_{\substack{x \in \Lambda_{2r_{f}L} \\ x_1,x_2,x_3,x_4\in \Lambda_{r_{f}L}}}\dfrac{\langle \sigma_x\sigma_{x_1}\rangle_\beta \langle \sigma_x\sigma_{x_2}\rangle_\beta \langle \sigma_x\sigma_{x_3}\rangle_\beta \langle \sigma_x\sigma_{x_4}\rangle_\beta}{\Sigma_L(\beta)^2}
    =\sum_{x\in \Lambda_{2r_f L}}\dfrac{\left(\sum_{y\in \Lambda_{r_f L}}\langle \sigma_x\sigma_y\rangle_\beta\right)^4}{\Sigma_L(\beta)^2}.
\end{equs}
Noticing that for $x \in \Lambda_{2r_f L}$, 
\begin{equs}
    \sum_{y\in \Lambda_{r_f L}}\langle \sigma_x\sigma_y\rangle_\beta\leq \chi_{3r_f L}(\beta)\leq \chi(\beta),
\end{equs}
one gets that for $C_2>0$,
\begin{equs}
    (1)
    \leq
    \chi(\beta)^4 \dfrac{|\Lambda_{2r_f L}|}{\Sigma_L(\beta)^2}
    \leq
    \frac{C_2}{L^d}.
\end{equs}
\textbf{Bound on (II)}. By \eqref{eq: spin decay} there exists $C_3>0$ such that for all $u,v\in \mathbb Z^d$ such that $u \neq v$,
\begin{equs}\label{upper bound subcritic}
    \langle\sigma_u\sigma_v\rangle_\beta\leq \dfrac{C_3}{|u-v|^{d+\boldsymbol{\alpha}}}.
\end{equs}
Thus, there exists $C_4,C_5,C_6>0$ such that
\begin{equs}
    (2)\leq \frac{C_4}{\Sigma_L(\beta)^2}L^{4d}\sum_{k\geq r_f L}\frac{k^{d-1}}{k^{4d+4\boldsymbol{\alpha}}}\leq C_5\frac{L^{4d}}{L^{3d+4\boldsymbol{\alpha}}\Sigma_L(\beta)^2}
    \leq 
    \frac{C_6}{L^{d+4\boldsymbol{\alpha}}}.
\end{equs}
This concludes the proof.
\end{proof}

\subsection{Computation of the covariance kernel}

We start with a useful lemma which estimates $\Sigma_L(\beta)$ as $L \rightarrow \infty$. 

\begin{lem}\label{lem: sigma} Let $d \geq 2$, $\boldsymbol{\alpha}, \mathbf{C} > 0$ in the interaction \eqref{eq: interaction}, and $\beta < \beta_c$. Then, for all $L \geq 1$,
\begin{equs}
\Sigma_L(\beta)=\chi(\beta)|\Lambda_L|\left(1+O(L^{-\frac{\boldsymbol{\alpha}}{\boldsymbol{\alpha}+1}})\right).
\end{equs}
\end{lem}

\begin{proof}
First note that, by translation-invariance,
\begin{equs}
\Sigma_L(\beta)
=
\sum_{x,y \in \Lambda_L} \langle \sigma_x \sigma_y \rangle_\beta
\leq
|\Lambda_L| \chi(\beta)
\end{equs}
where we recall $\chi(\beta) = \sum_{x \in \ZZ^d} \langle \sigma_0 \sigma_x \rangle_\beta$ is the susceptibility.

On the other hand, denoting $L(\boldsymbol{\alpha}):=L^{\frac{1}{\boldsymbol{\alpha}+1}}$, we have
\begin{equs}
\sum_{x,y \in \Lambda_L} \langle \sigma_x \sigma_y \rangle_\beta
\geq
|\Lambda_{L-L(\boldsymbol{\alpha})}| \chi_{L(\boldsymbol{\alpha})}(\beta)
\end{equs}
where $\chi_{k}(\beta) = \sum_{x \in \Lambda_{k}} \langle \sigma_0 \sigma_x \rangle_\beta$. By \eqref{eq: spin decay}, there exists $C>0$ such that
\begin{equs}
    \chi_{L(\boldsymbol{\alpha})}(\beta)
    =
    \chi(\beta)\frac{\chi_{L(\boldsymbol{\alpha})}}{\chi(\beta)}
    =
    \chi(\beta)\left(1-\frac{\chi(\beta)-\chi_{L(\boldsymbol{\alpha})}(\beta)}{\chi(\beta)}\right)
    \geq
    \chi(\beta)\left(1-CL^{-\frac{\boldsymbol{\alpha}}{\boldsymbol{\alpha}+1}}\right).
\end{equs}
Therefore,
\begin{equs}
|\Lambda_{L-L(\boldsymbol{\alpha})}| \chi_{L(\boldsymbol{\alpha})}(\beta)
\geq 
|\Lambda_L|\chi(\beta)\left(1-CL^{-\frac{\boldsymbol{\alpha}}{\boldsymbol{\alpha}+1}}\right),
\end{equs}
from which we conclude.
\end{proof}

We now show that the covariance of the smeared observables is given by their $L^2$-scalar product in the limit $L \rightarrow \infty$. The proof uses a summable polynomial upper bound on the spin-spin correlations and lattice symmetries of the model.

\begin{prop}\label{prop: covariance estimate}
Let $d \geq 2$,  $\boldsymbol{\alpha}, \mathbf{C} > 0$ in the interaction \eqref{eq: interaction}, and $\beta < \beta_c$. Then, for all $f,g\in \cC_c^\infty(\mathbb R^d)$ and for all $L \geq 1$,
\begin{equs}
\langle T_{f,L,\beta}(\sigma) T_{g,L,\beta}(\sigma) \rangle_{\beta}
=
\left(1+O(L^{-\frac{\boldsymbol{\alpha}}{\boldsymbol{\alpha}+1}})\right)\int_{\RR^d} f(\tilde x) g(\tilde x) \mathrm{d}\tilde x.
\end{equs}
\end{prop}

\begin{proof}
Using Lemma \ref{lem: sigma} above, $\Sigma_L(\beta)=|\Lambda_L|\chi(\beta)(1+O(L^{-\frac{\boldsymbol{\alpha}}{\boldsymbol{\alpha}+1}}))=2^dL^d\chi(\beta)(1+O(L^{-\frac{\boldsymbol{\alpha}}{\boldsymbol{\alpha}+1}}))$ as $L\rightarrow \infty$.

Assume that ${\rm supp} (f) \cup {\rm supp } (g) \subset B(0,R_1)$, which is clearly contained in the unit ball of radius $R_1$ in the $|\cdot|_\infty$ norm. Choose $R_2 > 2 R_1$. Then,
\begin{equs}
\langle T_{f,L,\beta}(\sigma) T_{g,L,\beta}(\sigma) \rangle_\beta
&=
\frac{2^d}{\Sigma_L(\beta)} \sum_{x,y \in \ZZ^d} f(x/L) g(y/L) \langle \sigma_x \sigma_y \rangle_\beta
\\
&=
\frac{2^d}{\Sigma_L(\beta)} \sum_{x,y \in \Lambda_{L R_1}} f(x/L) g(y/L) \langle \sigma_x \sigma_y \rangle_\beta
\\
&=
\frac{2^d}{\Sigma_L(\beta)} \sum_{x \in \Lambda_{L R_1}} f(x/L) \sum_{y \in \Lambda_{L R_2}(x)} g(y/L) \langle \sigma_x \sigma_y \rangle_\beta
\\
&=
\frac{2^d}{\Sigma_L(\beta)} \sum_{x \in \Lambda_{L R_1}} f(x/L) \sum_{y \in \Lambda_{L R_2}(x)} \Big( g(y/L) - g(x/L)\Big) \langle \sigma_x \sigma_y \rangle_\beta
\\
&\qquad + 
\frac{2^d}{\Sigma_L(\beta)} \sum_{x \in \Lambda_{L R_1}} f(x/L) g(x/L) \chi_{L R_2}(\beta).
\end{equs}
Now, by Taylor's theorem,
\begin{equs}
g(y/L)-g(x/L)
=
\nabla g(x/L) \cdot \left(\frac{y-x}{L}\right)+O\left(\frac{|y-x|^2}{L^2}\right).
\end{equs}
One has, by the invariance of the spin-spin correlations under lattice symmetries,
\begin{equs}
\sum_{y\in \Lambda_{LR_2}(x)}\nabla g(x/L)\cdot \left(\frac{y-x}{L}\right) \langle \sigma_x \sigma_y \rangle_\beta
=
0.
\end{equs}
By \eqref{eq: spin decay}, we may choose $\varepsilon=\varepsilon(\boldsymbol{\alpha})=\frac{\boldsymbol{\alpha}}{\boldsymbol{\alpha}+1}\in (0,\boldsymbol{\alpha} \wedge 1)$ and $C> 0$ such that
\begin{equs}
\langle \sigma_x \sigma_y \rangle_\beta
\leq
\frac{C}{|x-y|^{d+\varepsilon}}.
\end{equs}
Then, there exists $C'>0$ such that
\begin{equs}
\bigg|\frac{2^d}{\Sigma_L(\beta)} &\sum_{x \in \Lambda_{LR_1}} f(x/L) \sum_{y \in \Lambda_{LR_2}(x)} O\left( \frac{|x-y|^2}{L^2} \right) \langle \sigma_x \sigma_y \rangle_\beta \bigg|
\\
&\leq
\frac{2^d}{\Sigma_L(\beta)} \sum_{x \in \Lambda_{LR_1}} |f(x/L)| \sum_{y \in \Lambda_{LR_2}(x)} O\left( \frac{|x-y|^2}{L^2} \right) \frac{C}{|x-y|^{d+\varepsilon}}
\\
&\leq 
\frac{C'}{\Sigma_L(\beta)} \sum_{x \in \Lambda_{LR_1}} |f(x/L)| O(L^{-\varepsilon})
=
O(L^{-\varepsilon}).
\end{equs}

For the second term, by Riemann approximation and \eqref{eq: exact asymptotic}, 
\begin{equs}
\frac{2^d}{\Sigma_L(\beta)} \sum_{x \in \Lambda_{L R_1}} f(x/L) g(x/L) \chi_{L R_2}(\beta)
= 
(1+O(L^{-\varepsilon}))\int_{\R^d} f(\tilde x) g(\tilde x) \mathrm d\tilde x,
\end{equs}
from which the proof follows. 
\end{proof}

\subsection{Proof of white noise convergence}

\begin{proof}[Proof of Theorem \ref{thm: white noise}]
The proof of \eqref{eq: quantitative wn cv} follows directly from Propositions \ref{prop: triviality} and \ref{prop: covariance estimate}. Note that, by the linearity of the smeared observables, we have that, for any $f_1,\ldots,f_n\in \mathcal{C}^\infty_c(\RR^d)$, the vector $(T_{f_i,L,\beta}(\sigma))_{1\leq i\leq n}$ converges in distribution to a Gaussian vector distributed as $\{ W_0(f) : f \in \{f_1,\dots,f_n\} \}$. Thus, we have convergence of the set of smeared observables to white noise in finite dimensional distributions.
\end{proof}

\section{Emergence of fractional GFF correlations}

In this section, we prove Theorem \ref{thm: fgff} and Corollary \ref{cor: fgff conv}. We treat the cases $\boldsymbol{\alpha} \in \RR_{> 0} \setminus 2\NN^*$ and $\boldsymbol{\alpha} \in 2\NN^*$ separately. We start by showing that the expectation of the rescaled, renormalised product of smeared observables converges. This heavily exploits the exact asymptotic of the two-point function in Proposition \ref{prop: sharp asymptotics}. We combine this with Lemma \ref{lem: taylor pizetti} to show that the limiting expression agrees, up to a constant, with the covariance kernel of the fractional GFF. 

\begin{rem}\label{rem: spherical}
The computation exploits \textit{full} rotational symmetry of the interaction and this is the main reason why the Euclidean norm $|\cdot|$ is chosen in the interaction $J$. Otherwise, the choice of norm appears in the limit and does not agree with the fractional GFF kernel. In fact, it is not obvious that the limiting expression can be identified with a covariance kernel (it is not a priori symmetric or positive semi-definite).
\end{rem}

\subsection{Case $\boldsymbol{\alpha}\in \RR_{> 0} \setminus 2\NN^*$}

\begin{prop} \label{prop: case 1}
Let $d \geq 2$, $\boldsymbol{\alpha}\in \mathbb R_{> 0}\setminus 2\mathbb N^*$ and $\mathbf{C} > 0$ in the interaction \eqref{eq: interaction}, and $\beta < \beta_c$. Then, there exists $\kappa >0$ sufficiently small such that, for all $f,g \in \cC^\infty_c(\RR^d)$ and for all $L \geq 1$, 
\begin{equs}
L^{\boldsymbol{\alpha}} \Big\langle  \cR_{\boldsymbol{\alpha},\beta,L}\left[T_{f,L,\beta}(\sigma) T_{g,L,\beta}(\sigma)\right]\Big\rangle_\beta
=
\frac{\beta\chi(\beta)\mathbf{C}}{2}\tilde K_{\boldsymbol{\alpha},d}(f,g) \Big(1 + O(L^{-\kappa}) \Big),
\end{equs}
where
\begin{equs}
\tilde K_{\boldsymbol{\alpha},d}(f,g) 
=
\int_{\mathbb R^d \times \RR^d}f(\tilde x)|\tilde x-\tilde y|^{-d-\boldsymbol{\alpha}}\left(g(\tilde y)-{\rm Tay}_{\lfloor\boldsymbol{\alpha}\rfloor_2} g(\tilde y;\tilde x) \right)\mathrm{d}\tilde x\mathrm{d} \tilde y.
\end{equs}
\end{prop}

\begin{proof}
By adding and substracting the Taylor polynomial, we find
\begin{equs}
\langle T_{f,L,\beta}(\sigma) T_{g,L,\beta}(\sigma) \rangle_\beta 
&=
\frac{2^d}{\Sigma_L(\beta)} \sum_{x \in \ZZ^d} f(x/L) \sum_{y \in \ZZ^d \setminus \{x\}} \Big( g(y/L) - {\rm Tay}_{\lfloor\boldsymbol{\alpha}\rfloor_2}g(y/L;x/L)\Big) \langle \sigma_x \sigma_y \rangle_\beta
\\
&\qquad +
\frac{2^d}{\Sigma_L(\beta)} \sum_{x,y \in \ZZ^d} f(x/L) {\rm Tay}_{\lfloor\boldsymbol{\alpha}\rfloor_2}g(y/L;x/L) \langle \sigma_x \sigma_y \rangle_\beta.
\end{equs}
Note that
\begin{equs}
\Big\langle \cR_{\boldsymbol{\alpha},\beta,L}&\left[T_{f,L,\beta}(\sigma) T_{g,L,\beta}(\sigma)\right]\Big\rangle_\beta
\\
&=
\langle T_{f,L,\beta}(\sigma) T_{g,L,\beta}(\sigma) \rangle_\beta  - \frac{2^d}{\Sigma_L(\beta)} \sum_{x,y \in \ZZ^d} f(x/L) {\rm Tay}_{\lfloor\boldsymbol{\alpha}\rfloor_2}g(y/L;x/L) \langle \sigma_x \sigma_y \rangle_\beta.
\end{equs}
We analyse the remainder term. First, for $\varepsilon\in (0,1)$,
\begin{equs}
    T(L)
    &:=
    \frac{2^d}{\Sigma_L(\beta)} \sum_{x \in \ZZ^d} f(x/L) \sum_{y \in \ZZ^d \setminus \{x\}} \Big( g(y/L) - {\rm Tay}_{\lfloor\boldsymbol{\alpha}\rfloor_2}g(y/L;x/L) \Big) \langle \sigma_x \sigma_y \rangle_\beta 
    \\
    &=
    \frac{2^d}{\Sigma_L(\beta)} \sum_{x \in \ZZ^d} f(x/L) \sum_{y \in \Lambda_{L^{\varepsilon}}(x)} \Big( g(y/L) - {\rm Tay}_{\lfloor\boldsymbol{\alpha}\rfloor_2}g(y/L;x/L)\Big) \langle \sigma_x \sigma_y \rangle_\beta
    \\&\qquad
    + \frac{2^d}{\Sigma_L(\beta)} \sum_{x \in \ZZ^d} f(x/L) \sum_{y \in (\Lambda_{L^{\varepsilon}}(x))^c} \Big( g(y/L) - {\rm Tay}_{\lfloor\boldsymbol{\alpha}\rfloor_2}g(y/L;x/L)\Big) \langle \sigma_x \sigma_y \rangle_\beta 
    \\&  
    =:
    \frac{2^d}{\Sigma_L(\beta)} \sum_{x \in \ZZ^d} f(x/L) T_1(x,L) + f(x/L) T_2(x,L).
\end{equs}
Fix $x \in \ZZ^d$. Then, by Taylor's Theorem applied to $g(y/L)$ expanded around $x$ to order $\lfloor \boldsymbol{\alpha} \rfloor_2 + 2$, there exists $z(x) \in \RR^d$ such that
\begin{equs}
T_1(x,L)
&=
\sum_{y \in \Lambda_{L^{\varepsilon}}(x)} \Bigg( \frac{1}{(\lfloor \boldsymbol{\alpha} \rfloor_2 + 1)!} \nabla^{\lfloor \boldsymbol{\alpha} \rfloor_2 + 1}g(x/L) (y/L-x/L,\dots,y/L-x/L) 
\\
&\qquad \qquad  
+ \frac{1}{(\lfloor \boldsymbol{\alpha} \rfloor_2 + 2)!} \nabla^{\lfloor \boldsymbol{\alpha} \rfloor_2 + 2}g(z(x)/L) (y/L-x/L,\dots,y/L-x/L) \Bigg) \langle \sigma_x \sigma_y \rangle_\beta.
\end{equs}
By symmetry of the two-point function, 
\begin{equs}
\sum_{y \in \Lambda_{L^{\varepsilon}}(x)} \frac{1}{(\lfloor \boldsymbol{\alpha} \rfloor_2 + 1)!} \nabla^{\lfloor \boldsymbol{\alpha} \rfloor_2 + 1}g(x/L) (y/L-x/L,\dots,y/L-x/L) \langle \sigma_x \sigma_y \rangle_\beta
=
0.
\end{equs}
Hence, by the upper bound on spin correlations \eqref{eq: spin decay}, we have for some  $C_1>0$,
\begin{equs}
    \sup_{x \in \ZZ^d}|T_1(x,L)|
    &\leq
    \frac{C_1\| \nabla^{\lfloor \boldsymbol{\alpha} \rfloor_2+2} g\|_{L^\infty}}{L^{\lfloor \boldsymbol{\alpha} \rfloor_2+2}} \sum_{y \in \Lambda_{L^{\varepsilon}}(x)} \frac{1}{|x-y|^{d+\boldsymbol{\alpha}-\lfloor \boldsymbol{\alpha} \rfloor_2-2}}
    =
    O\Bigg(\frac{L^{\varepsilon(\lfloor \boldsymbol{\alpha} \rfloor_2+2-\boldsymbol{\alpha})}}{L^{\lfloor \boldsymbol{\alpha} \rfloor_2+2}} \Bigg)
\end{equs}
Thus, by the above displays together with Proposition \ref{prop: sharp asymptotics} and Lemma \ref{lem: sigma},  choosing $\kappa>0$ sufficiently small, we have that
\begin{equs}
 L^{\boldsymbol{\alpha}} &\Big\langle \cR_{\boldsymbol{\alpha},\beta,L}\left[T_{f,L,\beta}(\sigma) T_{g,L,\beta}(\sigma)\right]\Big\rangle_\beta
 \\
 &=
 \frac{\beta \chi(\beta)^2 \mathbf{C}}{2}\frac{2^d}{\Sigma_L(\beta) L^d} \sum_{x \in \ZZ^d} f(x/L) \sum_{y \in (\Lambda_{L^\epsilon}(x))^c}
 \frac{g(y/L) - {\rm Tay}_{\lfloor\boldsymbol{\alpha}\rfloor_2}g(y/L;x/L)}{|x/L-y/L|^{d+\boldsymbol{\alpha}}} \Big(1 + O(|x-y|^{-\delta}) \Big) 
 \\
 &\qquad \qquad +
 \frac{2^d}{\Sigma_L(\beta)} \sum_{x \in \ZZ^d} f(x/L)  O\Bigg(\frac{L^{\varepsilon(\lfloor \boldsymbol{\alpha} \rfloor_2+2-\boldsymbol{\alpha})}}{L^{\lfloor \boldsymbol{\alpha} \rfloor_2+2-\boldsymbol{\alpha}}} \Bigg)
 \\
 &=
 \frac{\beta \chi(\beta) \mathbf{C}}{2} \tilde K_{\boldsymbol{\alpha},d}(f,g) \Big(1 + O(L^{-\kappa}) \Big),
\end{equs}
which concludes the proof.
\end{proof}

\subsection{Case $\boldsymbol{\alpha}\in 2\NN^*$}

\begin{prop} \label{prop: case 2}
Let $d \geq 2$, $\boldsymbol{\alpha}\in 2\mathbb N^*$ and $\mathbf{C} > 0$ in the interaction \eqref{eq: interaction}, and let $\beta < \beta_c$. Then, there exists $\kappa>0$ sufficiently small such that, for all $f,g \in \cC^\infty_c(\RR^d)$ and for all $L \geq 1$
\begin{equs}
\frac{L^{\boldsymbol{\alpha}}}{\log L} \Big\langle \tilde \cR_{\boldsymbol{\alpha},\beta,L}\left[T_{f,L,\beta}(\sigma) T_{g,L,\beta}(\sigma)\right]\Big\rangle_\beta
=
(-1)^{\boldsymbol{\alpha}/2}(4\pi^2)^{\boldsymbol{\alpha}/2}H_{\boldsymbol{\alpha}/2}\frac{\beta\chi(\beta)\mathbf{C}}{2} K_{\boldsymbol{\alpha},d}(f,g) \Big(1 + O((\log L)^{-\kappa}) \Big)
\end{equs}
where $K_{\boldsymbol{\alpha},d}$ is defined in Definition \textup{\ref{def: fgff}}.
\end{prop}

\begin{proof}
We use the same notations as in the proof of Proposition \ref{prop: case 1}, but with $L^\varepsilon$ replaced by $\log L$ in the definitions of $T_1(x,L)$ and $T_2(x,L)$. By Taylor's Theorem (notice that in this case $\lfloor \boldsymbol{\alpha}\rfloor_2+2=\boldsymbol{\alpha}$) and symmetry considerations, we have for some $C_1>0$,
\begin{equs}
    \sup_{x \in \ZZ^d}|T_1(x,L)|
    &\leq
    \frac{C_1\| \nabla^{\boldsymbol{\alpha}} g\|_{L^\infty}}{L^{\boldsymbol{\alpha}}} \sum_{y \in \Lambda_{\log L}(x)} \frac{1}{|x-y|^{d}}
    =
    O\Big(\frac{\log\log L}{L^{\boldsymbol{\alpha}}} \Big).
\end{equs}
Hence, by Riemann approximation,
\begin{equs}
\frac{L^{\boldsymbol{\alpha}}}{\log L}\frac{2^d}{\Sigma_L(\beta)}\sum_{x \in \ZZ^d} f(x/L) T_1(x,L)
=
O\Big( \frac{\log\log L}{\log L} \Big).
\end{equs}

We now analyse
\begin{equs}
\frac{L^{\boldsymbol{\alpha}}}{\log L}\frac{2^d}{\Sigma_L(\beta)}\sum_{x \in \ZZ^d} f(x/L) T_2(x,L).
\end{equs}
Fix $x \in \ZZ^d$ and write
\begin{equs}
T_2(x,L)
&=
T_{2,1}(x,L) + T_{2,2}(x,L),
\end{equs}
where
\begin{equs}
T_{2,1}(x,L)
&:=
\sum_{y \in (\Lambda_{L}(x))^c} \Big(g(y/L) - {\rm Tay}_{\lfloor\boldsymbol{\alpha}\rfloor_2}g(y/L;x/L) \Big) \langle \sigma_x \sigma_y \rangle_\beta,
\\
T_{2,2}(x,L)
&:=
\sum_{y \in (\Lambda_{\log L}(x))^c \cap (\Lambda_{L}(x))} \Big(g(y/L) - {\rm Tay}_{\lfloor\boldsymbol{\alpha}\rfloor_2}g(y/L;x/L) \Big) \langle \sigma_x \sigma_y \rangle_\beta.
\end{equs}
By Taylor's Theorem, there exists $C_2>0$ such that for all $x,y\in \mathbb Z^d$,
\begin{equs}
\left|g(y/L) - {\rm Tay}_{\lfloor\boldsymbol{\alpha}\rfloor_2}g(y/L;x/L)\right|
\leq
C_2\| \nabla^{\boldsymbol{\alpha}-2}g\|_{L^\infty} |x-y|^{\boldsymbol{\alpha}-2}L^{-\boldsymbol{\alpha}+2}.
\end{equs}
Hence, using \eqref{eq: spin decay} and translation invariance,
\begin{equs}
\sup_{x\in \mathbb Z^d} |T_{2,1}(x,L)|\leq C_3L^{-\boldsymbol{\alpha}+2}\sum_{y\in \Lambda_L^c}|y|^{\boldsymbol{\alpha}-2}\langle \sigma_0\sigma_y\rangle_\beta=O(L^{-\boldsymbol{\alpha}}).
\end{equs}
Moreover, by Taylor's Theorem and symmetry considerations,
\begin{equs}
T_{2,2}(x,L)
&=
\sum_{y \in (\Lambda_{\log L}(x))^c \cap (\Lambda_{L}(x))} \Bigg(\frac{1}{\boldsymbol{\alpha}!} \nabla^{\boldsymbol{\alpha}} g(x/L) (y/L-x/L,\dots,y/L-x/L)
\\
&\qquad\qquad\qquad\qquad \qquad \qquad 
+O\left(\| \nabla^{\boldsymbol{\alpha} +2}g\|_{L^\infty} |x-y|^{\boldsymbol{\alpha}+2}L^{-(\boldsymbol{\alpha}+2)}\right)\Bigg) \langle \sigma_x \sigma_y \rangle_\beta
\\
&=
\sum_{y \in (\Lambda_{\log L}(x))^c \cap (\Lambda_{L}(x))} \frac{1}{\boldsymbol{\alpha}!} \nabla^{\boldsymbol{\alpha}} g(x/L) (y/L-x/L,\dots,y/L-x/L) \langle \sigma_x \sigma_y \rangle_\beta + O(L^{-{\boldsymbol{\alpha}}}).
\end{equs}
Hence,
\begin{equs}
\frac{L^{\boldsymbol{\alpha}}}{\log L} T_{2,2}(x,L)
&=
\sum_{y \in (\Lambda_{\log L}(x))^c \cap (\Lambda_{L}(x))} \frac{\beta \chi(\beta)^2 \mathbf{C}}{2 \boldsymbol{\alpha}!} 
\\
&\qquad \qquad
\times \frac{1}{\log L}\nabla^{\boldsymbol{\alpha}} g(x/L) (y-x,\dots,y-x) \frac{1+O(|x-y|^{-\delta})}{|x-y|^{d+\boldsymbol{\alpha}}} + O((\log L)^{-1})
\\
&=
\frac{\beta \chi(\beta)^2 \mathbf{C}}{2 } \Big(1+O((\log L)^{-\delta})\Big)
\\
 &\qquad \qquad
 \times \sum_{\substack{ \gamma \in \NN^d \\|\gamma|_1=\boldsymbol{\alpha}}}\frac{1}{\gamma!}\partial^\gamma g(x/L)\frac{1}{\log L}\sum_{y \in (\Lambda_{\log L}(x))^c \cap (\Lambda_{L}(x))} \frac{(y-x)^\gamma}{|x-y|^{d+\boldsymbol{\alpha}}} + O((\log L)^{-1}),
\end{equs}
where in the second equality we use that, for $\tilde z \in \mathbb R^d$ and $j\geq 0$,
\begin{equs}
\frac{1}{j!}\nabla^{j} g(\tilde x)(\tilde z, \ldots, \tilde z)=\sum_{|\gamma|_1=j}\frac{1}{\gamma!}\partial^{\gamma}g(\tilde x)\tilde z^{\gamma},
\end{equs}
and where the sum is over multi-indices $\gamma=(\gamma_1,\ldots,\gamma_d)\in \mathbb N^d$,  $\partial^\gamma=\partial^{\gamma_1}_1\cdots \partial^{\gamma_d}_d$, $\tilde z^\gamma=\prod_{k=1}^d \tilde z_k^{\gamma_k}$, and $\gamma!:=\gamma_1!\cdots\gamma_d!$. Furthermore: i) note that by symmetry, we may restrict the above sum to $\gamma \in (2\NN)^d$; and, ii) up to an error of order $O(\log \log L/\log L)$, we can sum over $y\in (\Lambda_L(x))\setminus \lbrace x\rbrace$ above. 

For each $k \in \NN^*$ and $\tilde y \in \RR^d$, denote by
\begin{equs}
\SS(\tilde y,k)
=
\{ \tilde x \in \RR^d : k-1 \leq |\tilde x- \tilde y|< k\}
\end{equs}
the spherical shell (closed at the interior boundary and open at the exterior boundary) of radius $k$ and centred at $\tilde y \in \RR^d$. Observe that $\bigsqcup_{k}\{ \mathbb S(0,k) \cap \ZZ^d \}=\mathbb Z^d$.
Thus, for any $\gamma \in (2\NN)^d$ such that $|\gamma|_1=\boldsymbol{\alpha}$, we have
\begin{equs}
\sum_{y \in(\Lambda_L(x))\setminus \lbrace x\rbrace} \frac{\prod_{i=1}^d (y_i-x_i)^{\gamma_i}}{|x-y|^{d+\boldsymbol{\alpha}}}
&=
\sum_{r=1}^L \frac{1}{r}\frac{1}{r^{d-1}} \sum_{y \in  \mathbb S(x,r)\cap (\Lambda_L(x))} \prod_{i=1}^d \left(\frac{y_i-x_i}{|x-y|}\right)^{\gamma_i}
\\
&=
\log L \cdot \Omega_d\int_{\tilde y\in S(0,1)}\tilde y_1^{\gamma_1}\ldots \tilde y_d^{\gamma_d}\mathrm{d}\mu(\tilde y)\Big(1+O(L^{-1})\Big),
\end{equs}
where the last line is by Riemann approximation, $\mu$ is the normalised Haar measure on the sphere $S(0,1)$, and $\Omega_d$ is the surface area of $S(0,1)$, explicitly given by
\begin{equs}
\Omega_d
:=
\frac{2\pi^{d/2}}{\Gamma(d/2)}.
\end{equs}
We have thus obtained,
\begin{equs}
\frac{L^{\boldsymbol{\alpha}}}{\log L} T_{2,2}(x,L)
=
\frac{\beta \chi(\beta)^2 \mathbf{C}}{2}\Omega_d \sum_{|\gamma|_1=\boldsymbol{\alpha}}\frac{1}{\gamma!}\partial^\gamma g(x/L)\int_{\tilde y\in S(0,1)}\tilde y_1^{\gamma_1}\ldots \tilde y_d^{\gamma_d}\mathrm{d}\mu(\tilde y)\Big(1+O\left((\log L)^{-\delta}\right)\Big).
\end{equs}
Using symmetry considerations for the spherical integrals above as in the proof of Lemma \ref{lem: taylor pizetti}, we have
\begin{equs}
\frac{L^{\boldsymbol{\alpha}}}{\log L} T_{2,2}(x,L)
=\frac{\beta \chi(\beta)^2 \mathbf{C}}{2}H_{\boldsymbol{\alpha}/2}\Delta^{\boldsymbol{\alpha}/2}g(\tilde{x})\Big(1+O\left((\log L)^{-\delta}\right)\Big).
\end{equs}
Hence,
\begin{equs}
\frac{L^{\boldsymbol{\alpha}}}{\log L} \Big\langle \cR_{\boldsymbol{\alpha},\beta,L}\left[T_{f,L,\beta}(\sigma) T_{g,L,\beta}(\sigma)\right]\Big\rangle_\beta
&=\frac{\beta \chi(\beta)\mathbf{C}}{2}H_{\boldsymbol{\alpha}/2}\int_{\mathbb R^d}f(\tilde x)\Delta^{\boldsymbol{\alpha}/2}g(\tilde x)\mathrm{d}\tilde x\Big(1+O\left((\log L)^{-\delta}\right)\Big).
\end{equs}
\end{proof}

\subsection{Proof of emergence of fractional GFF correlations}

\begin{proof}[Proof of Theorem \textup{\ref{thm: fgff}}]
The proof of Theorem \ref{thm: fgff} is a direct consequence of Propositions \ref{prop: case 1} and \ref{prop: case 2}, together with Lemma \ref{lem: taylor pizetti}.
\end{proof}

\appendix
\section{Rotational symmetry and Taylor's theorem}

\begin{lem}\label{lem: taylor pizetti} Let $d\geq 2$ and $\boldsymbol{\alpha}>0$. Then, for any $f,g\in \cC^\infty_c(\mathbb R^d)$,
\begin{equs}
\int_{\mathbb R^d\times \mathbb R^d}&|\tilde x-\tilde y|^{-(d+\boldsymbol{\alpha})}\left(f(\tilde x)g(\tilde y)-f(\tilde x){\rm Tay}_{\lfloor\boldsymbol{\alpha}\rfloor_2}g(\tilde y;\tilde x)\right)\mathrm{d}\tilde x\mathrm{d}\tilde y
\\
&= 
\int_{\RR^d \times \RR^d} |\tilde x-\tilde y|^{-(d+\boldsymbol{\alpha})} \Big( f(\tilde x)g(\tilde y) - \sum_{j=0}^{\lfloor\boldsymbol{\alpha}\rfloor_2/2} f(\tilde x) H_j \Delta^j g(\tilde x) |\tilde x -\tilde y|^{2j} \Big)\mathrm{d} \tilde x\mathrm{d}\tilde y
\end{equs}
where the $H_j$'s are defined by \eqref{def: H_j}. 
\end{lem}
\begin{proof}
For $\tilde x\in \mathbb R^d$, 
\begin{multline*}
\int_{\mathbb R^d}|\tilde x-\tilde y|^{-(d+\boldsymbol{\alpha})}\left(g(\tilde y)-{\rm Tay}_{\lfloor\boldsymbol{\alpha}\rfloor_2}g(\tilde y;\tilde x)\right)\mathrm{d}\tilde y
\\=
\int_{\mathbb R^d}|\tilde x-\tilde y|^{-(d+\boldsymbol{\alpha})}\left(g(\tilde y)-\sum_{j=0}^{\lfloor \boldsymbol{\alpha}\rfloor_2/2}\frac{1}{(2j)!}\nabla^{2j}g(\tilde x)(\tilde y-\tilde x,\ldots, \tilde y-\tilde x)\right)\mathrm{d}\tilde y
\end{multline*}
If $\varepsilon>0$, notice that for $j\in \lbrace 0,\ldots,\lfloor \boldsymbol{\alpha}\rfloor_2/2\rbrace$,
\begin{multline*}
\frac{1}{(2j)!}\int_{B(x,\varepsilon)^c}|\tilde x-\tilde y|^{-(d+\boldsymbol{\alpha})}\nabla^{2j}g(\tilde x)(\tilde y-\tilde x,\ldots, \tilde y-\tilde x)\mathrm{d}\tilde y
\\=
\Omega_d\int_{\varepsilon}^\infty \frac{r^{d-1}}{r^{d+\boldsymbol{\alpha}}}\mathrm{d}r\sum_{|\gamma|_1=2j}\frac{1}{\gamma!}\partial_1^{\gamma_1}\cdots \partial_d^{\gamma_d}g(\tilde x)\int_{S(0,1)}\tilde z_1^{\gamma_1}\ldots \tilde z_d^{\gamma_d}\mathrm{d}\mu(\tilde z).
\end{multline*}
Note that in the sum above, by symmetry considerations we may restrict to multi-indices $\gamma$ such that $\gamma_k$ is even for all $1\leq k \leq d$. 

By a standard computation,
\begin{equs}
\int_{S(0,1)}\tilde z_1^{\gamma_1}\ldots \tilde z_d^{\gamma_d}\mathrm{d}\mu(\tilde z)=\frac{1}{ d(d+2)\cdots (d+2j-2)}\prod_{k=1}^d(\gamma_k-1)!!,
\end{equs}
where we recall that for $\ell\in \mathbb  N$ odd, $\ell!!=\ell(\ell-2)\cdots 1$. It remains to show that
\begin{equs}\label{eq: eq pizetti taylor}
\sum_{|\gamma|_1=2j, \: \gamma_k\textup{ even}}\frac{1}{\gamma!}\partial_1^{\gamma_1}\cdots \partial_d^{\gamma_d}g(\tilde x)\prod_{k=1}^d(\gamma_k-1)!!
=\frac{1}{2^j j!}\Delta^{j}g(\tilde{x}).
\end{equs}
Indeed, notice that for $|\gamma|_1=2j$ with $\gamma_k$ even for $1\leq k \leq d$,
\begin{equs}
\prod_{k=1}^d(\gamma_k-1)!!=\frac{1}{2^{j}}\prod_{k=1}^d\frac{\gamma_k !}{(\gamma_k/2)!},
\end{equs}
so that the left-hand side in \eqref{eq: eq pizetti taylor} rewrites 
\begin{equs}
\sum_{|\gamma|_1=2j, \: \gamma_k\textup{ even}}\frac{1}{\gamma!}\partial_1^{\gamma_1}\cdots \partial_d^{\gamma_d}g(\tilde x)\prod_{k=1}^d(\gamma_k-1)!!
=\frac{1}{2^j}\sum_{|\gamma|_1=2j, \: \gamma_k\textup{ even}}\partial_1^{\gamma_1}\cdots \partial_d^{\gamma_d}g(\tilde x)\prod_{k=1}^d\frac{1}{(\gamma_k/2)!}
\end{equs}
Finally, since
\begin{equs}
\Delta^{j}g(\tilde x)=\sum_{i_1,\ldots,i_j\in \lbrace 1,\ldots,d\rbrace}\partial^{2j}_{i_1i_1\ldots i_ji_j}g(\tilde x),
\end{equs}
we get
\begin{eqnarray*}
\sum_{|\gamma|_1=2j, \: \gamma_k\textup{ even}}\partial_1^{\gamma_1}\cdots \partial_d^{\gamma_d}g(\tilde x)\prod_{k=1}^d\frac{1}{(\gamma_k/2)!}
&=&\sum_{|\gamma|_1=2j, \: \gamma_k\textup{ even}}\partial_1^{\gamma_1}\cdots \partial_d^{\gamma_d}g(\tilde x)\binom{j}{\gamma_1/2,\ldots,\gamma_d/2}\frac{1}{j!}
\\&=&\frac{1}{j!}\sum_{|\gamma|_1=2j, \: \gamma_k\textup{ even}}\sum_{\substack{i_1,\ldots,i_j\in \lbrace 1,\ldots,d\rbrace\\|\lbrace 1\leq \ell \leq j, \: i_\ell=k\rbrace|=\gamma_k/2}}\partial^{2j}_{i_1i_1\ldots i_ji_j}g(\tilde{x})
\\&=&\frac{1}{j!}\Delta^jg(\tilde x).
\end{eqnarray*}
Using the definition of $H_j$ \eqref{def: H_j},
\begin{equs}
\int_{B(x,\varepsilon)^c}|\tilde x-\tilde y|^{-(d+\boldsymbol{\alpha})}\nabla^{2j}g(\tilde x)(\tilde y-\tilde x,\ldots, \tilde y-\tilde x)\mathrm{d}\tilde y
=
\int_{B(x,\varepsilon)^c}|\tilde x-\tilde y|^{-(d+\boldsymbol{\alpha})}H_j\Delta^jg(\tilde x)|\tilde x-\tilde y|^{2j}\mathrm{d}\tilde y.
\end{equs}
In order conclude, we write 
\begin{equs}
\int_{\mathbb R^d}|\tilde x-\tilde y|^{-(d+\boldsymbol{\alpha})}\left(g(\tilde y)-{\rm Tay}_{\lfloor\boldsymbol{\alpha}\rfloor_2}g(\tilde y;\tilde x)\right)\mathrm{d}\tilde y
&=
\lim_{\epsilon \downarrow 0} \int_{B(x,\epsilon)^c}|\tilde x-\tilde y|^{-(d+\boldsymbol{\alpha})}\left(g(\tilde y)-{\rm Tay}_{\lfloor\boldsymbol{\alpha}\rfloor_2}g(\tilde y;\tilde x)\right)\mathrm{d}\tilde y,
\end{equs}
use the considerations above, and take $\varepsilon$ to $0$.
\end{proof}

\endappendix

\bibliographystyle{alpha}
\bibliography{ref.bib}

\end{document}